\documentclass{amsart}

\usepackage{amssymb,stmaryrd,mathrsfs,dsfont,amsmath}
\usepackage{colonequals} 
\usepackage{tikz}
\usepackage{float}
\usepackage{hyperref}

\theoremstyle{plain}
\newtheorem{theorem}{Theorem}[section]
\newtheorem{lemma}[theorem]{Lemma}
\newtheorem{proposition}[theorem]{Proposition}
\newtheorem{corollary}[theorem]{Corollary}

\theoremstyle{definition}

\DeclareMathOperator{\N}{\mathbb{N}}
\DeclareMathOperator{\Z}{\mathbb{Z}}

\input xy

\usepackage{microtype}
\begin{document}
	
	
\title[On prime coprime graphs of certain finite groups]{On prime coprime graphs of certain finite groups}
	
\author[Ravi Ranjan]{\bfseries Ravi Ranjan}
\address{Department of Mathematics, Central University of South Bihar, Gaya--824236, Bihar, India}
\email{raviranjan23@cusb.ac.in}
\author[Shubh N. Singh]{\bfseries Shubh N. Singh}
\address{Department of Mathematics, Central University of South Bihar, Gaya--824236, Bihar, India}
\email{shubh@cub.ac.in}


\subjclass[2020]{05C25, 05C76}
\keywords{Prime coprime graph, Hamiltonian graph, Clique number, $H$-Join, Finite group}


	
\begin{abstract}
The prime coprime graph $\Theta(G)$ of a finite group $G$ is the graph whose vertex set is $G$ and any two distinct vertices are adjacent if the greatest common divisor of their orders is either $1$ or a prime. In this paper, we investigate Hamiltonicity, clique number, and vertex degree of $\Theta(G)$ for cyclic, dihedral, and dicyclic groups $G$. We establish that $\Theta(G)$ admits a $(k,1)$-partition for cyclic, dihedral, and dicyclic groups $G$ of specified orders.
\end{abstract}
	
\maketitle


\section{Introduction}

The study of graphs associated with groups has garnered significant interest in the last two decades; see \cite{camron-ijgt22}. Adhikari and Banerjee \cite{adhikari-nsjom22} recently introduced the \emph{prime coprime graph} $\Theta(G)$ of a finite group $G$ as the graph whose vertex set is $G$ and any two distinct vertices are adjacent if and only if the greatest common divisor (gcd) of their orders is either $1$ or a prime. This graph is alternatively known as the \emph{co-prime order graph} in some literature; see \cite{hao-2022, li-2022}. The \emph{coprime graph} of a finite group $G$, introduced in \cite{satta-09} as the \emph{order prime graph}, naturally forms a spanning subgraph of $\Theta(G)$.


\vspace{0.4mm}
Adhikari and Banerjee \cite{adhikari-nsjom22} characterized the completeness and Eulerian properties of $\Theta(G)$ for all finite groups $G$. Furthermore, they determined the domination number, diameter, and girth of $\Theta(G)$ for all finite groups $G$. In \cite{adhikari-nsjom22}, Hamiltonicity, planarity, vertex connectivity, and signless Laplacian spectra of $\Theta(G)$ were also investigated for specific finite groups $G$. Hao et al. \cite{hao-2022} classified all finite groups $G$ whose $\Theta(G)$ are planar. Furthermore, the vertex connectivity of $\Theta(G)$ 
was determined for finite cyclic, dihedral, and dicyclic groups $G$ \cite{hao-2022}. For a fixed integer $k\geq 1$, Li et al. \cite{li-2022} proved that there are finitely many finite groups $G$ whose $\Theta(G)$ have (non)orientable genus $k$. Many other interesting results on prime coprime graphs have been obtained in the literature; see \cite{li-2022, sehgal-jmcs21, sehgal-25, sehgal-ijpam2024, sehgal-jmcsc-21}.

\vspace{0.4mm}
The subsequent sections of this paper are organized as follows. In Section 2, we establish some notation and define key terminology related to graphs and groups.
In Section 3, We establish that $\Theta(G)$ admits a $(k,1)$-partition for cyclic, dihedral, and dicyclic groups $G$ of specified orders. In Section 4, we investigate Hamiltonicity of $\Theta(G)$ for every finite cyclic, dihedral, and dicyclic group $G$. In Sections 5 and 6, we determine the clique number and vertex degree of $\Theta(G)$, respectively, for groups $G$ belonging to the cyclic, dihedral, and dicyclic classes.


\section{Preliminaries and Notation}
 
Let $\N$ denote the set of all positive integers. Unless stated otherwise, assume $k, \ell, m, n\in \N$, and let $p$, $q$, $r$ denote distinct primes. We denote by $\gcd(a, b)$ the greatest common divisor (gcd) of $a, b\in \N$. The Euler's totient function of an integer $n\in \N$ is denoted by $\varphi(n)$. We denote the cardinality of a set $X$ by $|X|$. For any sets $A$ and $B$, let $A\setminus B \colonequals \{x\in A\colon x\notin B\}$. A \emph{partition} of a nonempty set $X$ is a family of nonempty, pairwise disjoint subsets of $X$ whose union is $X$. 


\vspace{0.4mm}

All graphs considered in this paper are finite, undirected, and simple. We denote by $K_n$ the complete graph, by $P_n$ the path, by $C_n$ the cycle, and by $E_n$ the graph with no edges, each one on $n$ vertices. Let $\Gamma$ be a graph. We use $V(\Gamma)$ and $E(\Gamma)$ to denote the vertex set and edge set of $\Gamma$, respectively. We use $uv$ to denote the edge of $\Gamma$ that joins the vertices $u$ and $v$ in $\Gamma$. The degree of a vertex $v\in V(\Gamma)$ is denoted by $\deg_{\Gamma}(v)$, and the minimum degree of $\Gamma$ is denoted by $\delta(\Gamma)$. The number of connected components of $\Gamma$ is denoted by $c(\Gamma)$. For a nonempty subset $U$ of $V(\Gamma)$, the subgraph of $\Gamma$ induced by $U$ is denoted by $\Gamma[U]$. The clique number and vertex connectivity of $\Gamma$ are denoted by $\omega(\Gamma)$ and $\kappa(\Gamma)$, respectively. A vertex of $\Gamma$ is a {\emph dominating vertex} if it is adjacent to every other vertex in $\Gamma$.  Let $\Gamma_1$ and $\Gamma_2$ be two vertex-disjoint subgraphs of $\Gamma$. We write $\Gamma_1 \sim \Gamma_2$ to denote that every vertex of $\Gamma_1$ is adjacent to every vertex of $\Gamma_2$ in the graph $\Gamma$. If there is no edge in $\Gamma$ that joins a vertex in $\Gamma_1$ to a vertex in $\Gamma_2$, then we say that $\Gamma_1$ is \emph{independent} of $\Gamma_2$. The notation $\Gamma \cong \Gamma'$ denotes that the graphs $\Gamma$ and $\Gamma'$ are isomorphic. The \emph{join} of any two vertex-disjoint graphs $\Gamma$ and $\Gamma'$, denoted by $\Gamma \vee \Gamma'$, is the graph whose vertex and edge sets are $V(\Gamma)\cup V(\Gamma')$ and $E(\Gamma)\cup E(\Gamma') \cup \{uv\colon u\in V(\Gamma), v\in V(\Gamma')\}$, respectively.

\vspace{0.4mm}
All groups considered in this paper are  finite. The cyclic group of order $n$ is denoted by $\Z_n = \{0,1,\ldots, n-1\}$. Let $G$ be a group. An element of order $2$ in $G$ is an \emph{involution} of $G$. For an element $g\in G$, we write $|g|$ for the order of $g$ and $\langle g \rangle$ for the subgroup generated by $g$.  We use $H \leq G$ to denote that $H$ is a subgroup of $G$. Let $S(G) \colonequals \{g\in G\colon|g| \text{ is either } 1 \text{ or a prime}\}$ and $T(G) \colonequals G \setminus S(G)$. A group is an \emph{EPO-group} if all of its nonidentity elements are of prime order.



\vspace{0.4mm}

For terminology and notation pertaining to graphs, groups, and number theory not defined herein, we refer the reader to \cite{west-b01}, \cite{gallian-b21}, and \cite{burton-b07}, respectively. To conclude this section, we present several established results from the literature.

\begin{theorem}\cite[Theorem~4.4]{gallian-b21}\label{th_numb-of-elem-cyclic-grp}
If $d$ is a positive divisor of $n$, then the number of elements of
order $d$ in a cyclic group of order $n$ is $\varphi(d)$.
\end{theorem}


\begin{theorem}\cite[Theorem~7.3]{burton-b07}\label{th_euler-phi-formula}
If $n = p_1^{k_1} p_2^{k_2} \cdots p_t^{k_t}$ is the canonical factorization of an integer $n>1$, then 
\[\varphi(n) = \left(p_1^{k_1}- p_1^{k_1-1}\right) \left(p_2^{k_2}- p_2^{k_2-1}\right) \cdots \left(p_t^{k_t}- p_t^{k_t-1}\right).\]
\end{theorem}


\begin{theorem} \cite[Theorem~3.1(b)]{adhikari-nsjom22}\label{th_singletion-dominating-set}
Let $G$ be a group and $g\in G$. Then $g$ is a dominating vertex of $\Theta(G)$ if and only if $|g|$ is either $1$ or a prime. 
\end{theorem}


\begin{theorem} \cite[Theorem~3.5]{adhikari-nsjom22} \label{th_group-EPO-iff-complete}
Let $G$ be a group of order $n$. Then $\Theta(G)\cong K_n$ if and only if $G$ is an EPO-group.
\end{theorem}

\newpage
\section{$(k, 1)$-graphs}
For any two nonnegative integers $k, \ell$ with $k+\ell > 0$, recall that a graph is a \emph{$(k, \ell)$-graph} if its vertex set admits a partition into $k$ independent sets and $\ell$ cliques; see \cite[p.~48]{brand-dm96}. Such a partition is called a \emph{$(k, \ell)$-partition} of the graph. In particular, a $(1,1)$-graph is known as a \emph{split graph}; see \cite[p.~667]{hammer-cjm77}.

\vspace{0.5mm}
This section comprises three subsections, each proving that the graph $\Theta(G)$ admits a $(k,1)$-partition for specific cyclic, dihedral, and dicyclic groups $G$. Before that, we recall the definition of $H$-join of vertex-disjoint graphs, as delineated in \cite{cardoso-dm2013}. Let $H$ be a graph with $V(H) = \{1, 2, \ldots, k\}$, where $k\geq 2$, and let $\{\Gamma_1, \Gamma_2, \ldots, \Gamma_k\}$ be a family of pairwise vertex-disjoint graphs. The \emph{$H$-join} of $\Gamma_1, \Gamma_2, \ldots, \Gamma_k$, denoted by $H[\Gamma_1, \Gamma_2, \ldots, \Gamma_k]$, is the graph whose vertex and edge sets are
\[ \bigcup_{i =1}^k V(\Gamma_i)  \quad\text{ and }\quad \Big(\displaystyle\bigcup_{i =1}^k E(\Gamma_i)\Big) \cup \Big( \displaystyle\bigcup_{ij\in E(H)} \{uv\colon u\in V(\Gamma_i), v\in V(\Gamma_j)\}\Big),\]  
respectively. It is clear that $\Gamma_1 \vee \Gamma_2 = P_2[\Gamma_1, \Gamma_2]$.

%

\vspace{0.4mm}
By Theorem~\ref{th_singletion-dominating-set}, the graph $\Theta(G)$ admits a $(0, 1)$-partition for every group $G$ of prime order.

\vspace{0.4mm}

In the following proposition, we establish that the graph $\Theta(G)$ of any group $G$ of order $pq$ admits either a $(1,1)$-partition or a $(0,1)$-partition depending on $G$ is cyclic or noncyclic, respectively.

\begin{proposition}\label{pr_group-of-2-distnct-prime}
If $G$ is a group of order $pq$, then
	\begin{equation*}
		\Theta(G)\cong
		\begin{cases}
			K_{p+q-1}\vee E_{(p-1)(q-1)} & \text{if $G$ is cyclic},\\
			K_{pq} & \text{if $G$ is noncyclic}.
		\end{cases}
	\end{equation*}
\end{proposition}

\begin{proof}[\textbf{Proof}]
Assume that $G$ is a group of order $pq$. If $G$ is noncyclic, then $G$ is an EPO-group. Therefore $\Theta(G)\cong K_{pq}$ by Theorem~\ref{th_group-EPO-iff-complete}. So, assume further that $G$ is cyclic. By combining Theorems~\ref{th_numb-of-elem-cyclic-grp} and \ref{th_euler-phi-formula}, we obtain $|S(G)| = p+q-1$ and $|T(G)| = (p-1)(q-1)$. Now observe that the gcd of the orders of any two distinct elements in $T(G)$ is a composite number, and so $\Theta(G)[T(G)] \cong E_{(p-1)(q-1)}$. By Theorem~\ref{th_singletion-dominating-set}, we thus obtain $\Theta(G)[S(G)] \cong K_{p+q-1}$ and $\Theta(G)[S(G)] \sim \Theta(G)[T(G)]$. Hence $\Theta(G)\cong K_{p+q-1}\vee E_{(p-1)(q-1)}$.
\end{proof}


\subsection{Cyclic Groups}
In this subsection, we prove that the graph $\Theta(\Z_n)$ admits the structure of a $(k, 1)$-graph for each $n \in \{p^m,\ pq^m,\ p^\ell q^m,\ pqr\}$, where $\ell, m \geq 2$.

\begin{proposition}\label{pr_H-join-cyclic-pm}
For all $m\geq 2$, $\Theta(\Z_{p^m})\cong K_{p}\vee E_{p^m-p}$. Consequently, $\Theta(\Z_{p^m})$ is a split graph.
\end{proposition}

\begin{proof}[\textbf{Proof}]
Assume that $m\geq 2$, and set $n= p^m$. By combining Theorems~\ref{th_numb-of-elem-cyclic-grp} and \ref{th_euler-phi-formula}, we obtain $|S(\Z_{n})| = p$ and $|T(\Z_{n})|=n-p$. Now observe that the gcd of the orders of any two distinct elements in $T(\Z_{n})$ is a composite number, and so $\Theta(\Z_{n})[T(\Z_{n})] \cong E_{n-p}$. By Theorem~\ref{th_singletion-dominating-set}, we thus obtain $\Theta(\Z_{n})[S(\Z_{n})] \cong K_{p}$ and $\Theta(\Z_{n})[S(\Z_{n})] \sim  \Theta(\Z_{n})[T(\Z_{n})]$. Hence $\Theta(\Z_{n})\cong K_{p}\vee E_{n-p}$.
\end{proof}

\begin{proposition}\label{pr_H-join-cyclic-pqm}
For all $m \geq 2$, 
\[\Theta(\Z_{pq^m})\cong H[K_{p+q-1},E_{q^m-q},E_{(p-1)(q-1)},E_{(p-1)(q^m-q)}],\] where $H$ is given in Figure~\ref{graph-H-of-cyclic-pqm}. 
Consequently, $\Theta(\Z_{pq^m})$ is a $(3,1)$-graph.
	\vspace{-2mm}
	\begin{figure}[h]
		\begin{tikzpicture}
			\draw[fill=black] (0,0) circle (1.2pt);
			\draw[fill=black] (-0.7,0) circle (1.2pt);
			\draw[fill=black] (0.8,0.4) circle (1.2pt);
			\draw[fill=black] (0.8,-0.4) circle (1.2pt);
			
			\draw[line width=0.6 pt] (-0.7,0) -- (0,0) -- (0.8,0.4)--(0.8, -0.4)--(0,0);
			
			\node at (0,0.2) {$1$};
			\node at (0.95, 0.4) {$2$};
			\node at (0.95, -0.4) {$3$};
			\node at (-0.7, 0.2) {$4$};
		\end{tikzpicture}	
		\vspace{-3mm}
		\caption{Graph $H$}	\label{graph-H-of-cyclic-pqm}
	\end{figure}
\end{proposition}

\begin{proof}[\textbf{Proof}]
Assume that $m\geq 2$. Set $n = pq^m$, and define
	\begin{align*}
		S_1&= S(\Z_{n}) & S_3&=\{g\in \Z_n \colon |g|=pq\}\\
		S_2&=\{g\in \Z_n \colon |g|=q^j,\; 2\leq j \leq  m\}
		  & S_4&=\{g\in \Z_n \colon |g|=pq^j,\; 2\leq j \leq  m\}.
	\end{align*}
By construction, the collection $\{S_1, S_2, S_3, S_4\}$ forms a partition of $\Z_{n}$. By combining Theorems~\ref{th_numb-of-elem-cyclic-grp} and \ref{th_euler-phi-formula}, we obtain $|S_1|=p+q-1$, $|S_2|=q^m-q$, $|S_3|=(p-1)(q-1)$, and $|S_4|=(p-1)(q^m-q)$. Now observe for each $i \in \{2, 3, 4\}$ that the gcd of the orders of any two distinct elements in $S_i$ is a composite number. Therefore $\Gamma_2 \colonequals \Theta(\Z_{n})[S_2]\cong E_{q^m-q}$, $\Gamma_3 \colonequals \Theta(\Z_{n})[S_3]\cong E_{(p-1)(q-1)}$, and $\Gamma_4 \colonequals \Theta(\Z_{n})[S_4]\cong E_{(p-1)(q^m-q)}$. By Theorem~\ref{th_singletion-dominating-set}, we thus obtain $\Gamma_1 \colonequals \Theta(\Z_{n})[S_1]\cong K_{p+q-1}$ and $\Gamma_1 \sim \Gamma_i$ for each $i\in \{2,3,4\}$. Furthermore, observe that $\Gamma_2 \sim \Gamma_3$. Finally, observe that $\Gamma_4$ is independent of both $\Gamma_2$ and $\Gamma_3$. Hence $\Theta(\Z_n)\cong H[\Gamma_1, \Gamma_2, \Gamma_3, \Gamma_4]$, where $H$ is given in Figure~\ref{graph-H-of-cyclic-pqm}.
\end{proof}


\begin{proposition}\label{pr_H-join-cyclic-pmqn}
For all $\ell, m \ge 2$, the graph $\Theta(\Z_{p^\ell q^m})$ is isomorphic to the $H$-join \[H[K_{p+q-1}, E_{p^\ell-p}, E_{q^m-q}, E_{(p-1)(q-1)}, E_{(p^\ell-p)(q-1)}, E_{(q^m-q)(p-1)}, E_{(p^\ell-p)(q^m-q)}],\] where $H$ is given in Figure~\ref{graph-H-of-cyclic-pmqn}. 
Consequently, $\Theta(\Z_{p^\ell q^m})$ is a $(6,1)$-graph.
	\vspace{-2mm}
	\begin{figure}[h]
		\begin{tikzpicture}[very thick]
			
			\draw[fill=black] (0,0) circle (1.2pt);
			\draw[fill=black] (2,-0.75) circle (1.2pt);
			\draw[fill=black] (2,0.75) circle (1.2pt);
			\draw[fill=black] (1,0) circle (1.2pt);
			\draw[fill=black] (-1,0.75) circle (1.2pt);
			\draw[fill=black] (-1,-0.75) circle (1.2pt);
			\draw[fill=black] (-1,0) circle (1.2pt);
			\draw[line width=0.6 pt] (0,0)--(-1,-0.75)--(2,-0.75)--(2,0.75)--(-1,0.75)--(0,0);
			\draw[line width=0.6 pt] (-1,0)--(0,0)--(1,0);
			\draw[line width=0.6 pt] (0,0)--(2,-0.75)--(1,0)--(2,0.75)--(0,0);

			\node at (0,0.25)  {$1$};
			\node at (2.2,-0.75) {$2$};
			\node at (2.2,0.75) {$3$};
			\node at (1.3,0) {$4$};
			\node at (-1.2,0.75)  {$5$};
			\node at (-1.2,-0.75) {$6$};
			\node at (-1.2,0) {$7$};
		\end{tikzpicture}
		\vspace{-3mm}
		\caption{Graph $H$}	\label{graph-H-of-cyclic-pmqn}
	\end{figure}
\end{proposition}

\begin{proof}[\textbf{Proof}]
Assume that $\ell, m\geq 2$. Set $n= p^\ell q^m$, and define	
	\begin{align*}
		S_1&=S(\Z_n) & S_5&=\{g\in \Z_n \colon |g|=p^jq,\; 2\leq j \leq \ell\}\\
		S_2&=\{g\in \Z_n \colon |g|=p^j,\; 2\leq j \leq  \ell\}& S_6&=\{g\in \Z_n \colon  |g|=pq^k,\;  2\leq k \leq m\}\\
		S_3&=\{g\in \Z_n \colon  |g|=q^k,\; 2\leq k \leq  m\} & S_7 &=\{g\in \Z_n \colon  |g|=p^jq^k,\; 2\leq j \leq \ell,\; 2\leq k \leq m\}. \\
		S_4&=\{g\in \Z_n \colon |g|=pq\}
	\end{align*}
By construction, the collection $\{S_1, S_2, \ldots, S_7\}$ forms a partition of $\Z_{n}$. By combining Theorems~\ref{th_numb-of-elem-cyclic-grp} and \ref{th_euler-phi-formula}, we obtain $|S_1|=p+q-1$, $|S_2|=p^\ell-p$, $|S_3|=q^m-q$, $|S_4|=(p-1)(q-1)$, $|S_5|=(p^\ell-p)(q-1)$, $|S_6|=(q^m-q)(p-1)$, $|S_7|=(p^\ell-p)(q^m-q)$. Now observe for each $i \in \{2, 3,\ldots,7\}$ that the gcd of the orders of any two distinct elements in $S_i$ is a composite number. Therefore $\Gamma_2 \colonequals \Theta(\Z_n)[S_2]\cong E_{p^\ell-p}$, $\Gamma_3 \colonequals \Theta(\Z_n)[S_3]\cong E_{q^m-q}$, $\Gamma_4 \colonequals  \Theta(\Z_n)[S_4]\cong E_{(p-1)(q-1)}$, $\Gamma_5 \colonequals \Theta(\Z_n)[S_5]\cong E_{(p^\ell-p)(q-1)}$, $\Gamma_6 \colonequals  \Theta(\Z_n)[S_6]\cong E_{(p-1)(q^m-q)}$, and $\Gamma_7 \colonequals \Theta(\Z_n)[S_7]\cong E_{(p^\ell-p)(q^m-q)}$. By Theorem~\ref{th_singletion-dominating-set}, we thus obtain $\Gamma_1 \colonequals \Theta(\Z_n)[S_1]\cong K_{p+q-1}$ and $\Gamma_1 \sim \Gamma_i$ for each $i\in \{2,3,\ldots,7\}$. Furthermore, observe that $\Gamma_2 \sim \Gamma_3$, $\Gamma_2 \sim \Gamma_4$, $\Gamma_2 \sim \Gamma_6$, $\Gamma_3 \sim \Gamma_4$, and $\Gamma_3 \sim \Gamma_5$. Finally, observe that $\Gamma_7$ is independent of $\Gamma_i$ for each $i \in \{2, 3, \ldots, 6\}$; $\Gamma_6$ is independent of $\Gamma_i$ for each $i \in \{3, 4, 5\}$; and $\Gamma_5$ is independent of both $\Gamma_2$ and $\Gamma_4$. Hence $\Theta(\Z_n)\cong H[
\Gamma_1, \Gamma_2, \ldots, \Gamma_7]$, where $H$ is given in Figure~\ref{graph-H-of-cyclic-pmqn}.
\end{proof}

\begin{proposition} \label{pr_H-join-cyclic-pqr}
We have
	\[\Theta(\Z_{pqr})\cong H[K_{p+q+r-2},E_{(p-1)(q-1)},E_{(q-1)(r-1)},E_{(p-1)(r-1)}, E_{(p-1)(q-1)(r-1)}],\] where $H$ is given in Figure~\ref{graph-H-of-cyclic-pqr}.
Consequently, $\Theta(\Z_{pqr})$ is a $(4,1)$-graph.
	\vspace{-2mm}
	\begin{figure}[h]
		\begin{tikzpicture}
			
			\draw[fill=black] (0,0) circle (1.2pt);
			\draw[fill=black] (0.75,0) circle (1.2pt);
			\draw[fill=black] (-0.75,0.75) circle (1.2pt);
			\draw[fill=black] (-0.75,-0.75) circle (1.2pt);
			\draw[fill=black] (1.5,0) circle (1.2pt);
			
			\draw[line width=0.6 pt] (0,0)--(0.75,0)--(-0.75,0.75)--(-0.75,-0.75)--(0.75,0);
			\draw[line width=0.6 pt] (-0.75,0.75)--(0,0)--(-0.75,-0.75);
			\draw[line width=0.6 pt] (0.75,0)--(1.5,0);

			\node at (0.85,0.2)  {$1$};
			\node at (-0.25,0) {$2$};
			\node at (-0.9,0.75) {$3$};
			\node at (-0.9,-0.75) {$4$};
			\node at (1.5,0.2) {$5$};
		\end{tikzpicture}
		\vspace{-3mm}
		\caption{Graph $H$}	\label{graph-H-of-cyclic-pqr}
	\end{figure}
\end{proposition}

\begin{proof}[\textbf{Proof}]
Set $n = pqr$, and define
	\begin{align*}
		S_1&= S(\Z_n)& S_4 & =\{g\in \Z_n \colon |g|=pr\}\\
		S_2& =\{g\in \Z_n \colon |g|=pq\} &S_5&=\{g\in \Z_n \colon  |g|=pqr\}.\\
		S_3&=\{g\in \Z_n \colon |g|=qr\}
	\end{align*}
By construction, the collection $\{S_1, S_2, \ldots, S_5\}$ forms a partition of $\Z_{n}$. By combining Theorems~\ref{th_numb-of-elem-cyclic-grp} and \ref{th_euler-phi-formula}, we obtain $|S_1|=p+q+r-2$, $|S_2|=(p-1)(q-1)$, $|S_3|=(q-1)(r-1)$, $|S_4|=(p-1)(r-1)$, and $|S_4|=(p-1)(q-1)(r-1)$. Now observe for each $i \in \{2, 3, 4, 5\}$ that the gcd of the orders of any two distinct elements in $S_i$ is a composite number. Therefore $\Gamma_2 \colonequals \Theta(\Z_n)[S_2]\cong E_{(p-1)(q-1)}$, $\Gamma_3 \colonequals \Theta(\Z_n)[S_3]\cong E_{(q-1)(r-1)}$, $\Gamma_4 \colonequals \Theta(\Z_n)[S_4]\cong E_{(p-1)(r-1)}$, and $\Gamma_5 \colonequals \Theta(\Z_n)[S_5]\cong E_{(p-1)(q-1)(r-1)}$. By Theorem~\ref{th_singletion-dominating-set}, we thus obtain $\Gamma_1 \colonequals \Theta(\Z_n)[S_1]\cong K_{p+q+r-2}$ and $\Gamma_1 \sim \Gamma_i$ for each $i\in \{2,3,4,5\}$. Furthermore, observe that $\Gamma_2 \sim \Gamma_3$, $\Gamma_2 \sim \Gamma_4$, and $\Gamma_3 \sim \Gamma_4$. Finally, observe that $\Gamma_5$ 
 is independent of each of the subgraphs $\Gamma_2$, $\Gamma_3$, and $\Gamma_4$.
 Hence $\Theta(\Z_n)\cong H[\Gamma_1, \Gamma_2, \ldots, \Gamma_5]$, where $H$ is given in Figure~\ref{graph-H-of-cyclic-pqr}.
\end{proof}

\subsection{Dihedral Groups}

For every integer $n\geq 3$, recall that the \emph{dihedral group} $D_{n}$ of order $2n$ is a nonabelian group having the presentation: \[D_{n} = \langle a, b \colon a^n =  b^2=1, \ b^{-1}ab = a^{-1}\rangle,\] where $1$ is the identity of $D_{n}$ (cf. \cite[p.~181]{james-b01}). 

\vspace{0.4mm}
In this subsection, we prove that the graph $\Theta(D_n)$ admits the structure of a $(k, 1)$-graph for each $n \in \{p, pq, p^m,\ pq^m,\ p^\ell q^m,\ pqr\}$, 
where $\ell, m \geq 2$.


\begin{proposition}
The graph $\Theta(D_{p})$ is isomorphic to $K_{2p}$.
\end{proposition}
\begin{proof}[\textbf{Proof}]
It follows directly from Theorem~\ref{th_group-EPO-iff-complete}, since $D_p$ is an EPO-group of order $2p$.
\end{proof}

\begin{proposition}
We have $\Theta(D_{pq})\cong K_{pq+p+q-1}\vee E_{(p-1)(q-1)}$. Consequently, $\Theta(D_{pq})$ is a split graph.
\end{proposition}
\begin{proof}[\textbf{Proof}]
Set $n = pq$. Observe that $|S(D_{n})| = n+p+q-1$, and so $|T(D_{n})| = (p-1)(q-1)$. Furthermore, observe that the gcd of the orders of any two distinct elements in $T(D_{n})$ is a composite number, and so $\Theta(D_{n})[T(D_{n})] \cong E_{(p-1)(q-1)}$.  
By Theorem~\ref{th_singletion-dominating-set}, we thus obtain $\Theta(D_{n})[S(D_{n})] \cong K_{n+p+q-1}$ and $\Theta(D_{n})[S(D_{n})] \sim \Theta(D_{n})[T(D_n)]$. Hence $\Theta(D_{n})\cong K_{n+p+q-1}\vee E_{(p-1)(q-1)}$.
\end{proof}


\begin{proposition}
For all $m\geq 2$, $\Theta(D_{p^m})\cong K_{p^m+p}\vee E_{p^m-p}$. Consequently, $\Theta(D_{p^m})$ is a split graph.
\end{proposition}
\begin{proof}[\textbf{Proof}]
Assume that $m\geq 2$, and set $n = p^m$. Observe that $|S(D_n)| = n+p$, and so  $|T(D_n)| = n-p$. Furthermore, observe that the gcd of the orders of any two distinct elements in $T(D_n)$ is a composite number, and so $\Theta(D_{n})[T(D_n)] \cong E_{n-p}$. By Theorem~\ref{th_singletion-dominating-set}, we thus obtain $\Theta(D_n)[S(D_n)] \cong K_{n+p}$ and $\Theta(D_n)[S(D_n)] \sim \Theta(D_{n})[T(D_n)]$. Hence $\Theta(D_{n})\cong K_{n+p}\vee E_{n-p}$.
\end{proof}

\begin{proposition}
For all $m\ge 2$,
\[\Theta(D_{pq^m})\cong H[K_{pq^m+p+q-1},E_{q^m-q},E_{(p-1)(q-1)},E_{(p-1)(q^m-q)}],\] where $H$ is given in Figure~\ref{graph-H-of-cyclic-pqm}. Consequently, $\Theta(D_{pq^m})$ is a $(3, 1)$-graph.
\end{proposition}

\begin{proof}[\textbf{Proof}]
Assume that $m\geq 2$. Set $n = pq^m$, and define
	\begin{align*}
		S_1 &=S(D_{n}) &  S_3 &=\{g\in D_{n} \colon |g|=pq\} \\
		S_2 &=\{g\in D_{n} \colon |g|=q^j,\ 2\leq j \leq  m\}
		 & S_4&=\{g\in D_{n} \colon  |g|=pq^j,\  2\leq j \leq m\}.
	\end{align*}
By construction, the collection $\{S_1, S_2, S_3, S_4\}$ forms a partition of $D_{n}$. By combining Theorems~\ref{th_numb-of-elem-cyclic-grp} and \ref{th_euler-phi-formula}, we obtain
 $|S_1| = n+p+q-1$, $|S_2| = q^m-q$, $|S_3|=(p-1)(q-1)$, and $|S_4|=(p-1)(q^m-q)$. Now observe for each $i\in \{2,3,4\}$ that the gcd of the orders of any two distinct elements in $S_i$ is a composite number. Therefore  
$\Gamma_2 \colonequals \Theta(D_n)[S_2] \cong E_{q^m-q}$, $\Gamma_3 \colonequals \Theta(D_n)[S_3] \cong E_{(p-1)(q-1)}$, and $\Gamma_4 \colonequals \Theta(D_n)[S_4] \cong E_{(p-1)(q^m-q)}$. By Theorem~\ref{th_singletion-dominating-set}, we thus obtain $\Gamma_1 \colonequals \Theta(D_n)[S_1] \cong K_{n+p+q-1}$ and $\Gamma_1 \sim \Gamma_i$ for each $i\in \{2,3,4\}$. Furthermore, observe that $\Gamma_2 \sim \Gamma_3$. Finally, observe that $\Gamma_4$ is independent of both $\Gamma_2$ and $\Gamma_3$. Hence $\Theta(D_n)\cong H[\Gamma_1, \Gamma_2, \Gamma_3, \Gamma_4]$, where $H$ is given in Figure~\ref{graph-H-of-cyclic-pqm}.
\end{proof}

\begin{proposition}
For all $\ell, m\geq 2$, the graph $\Theta(D_{p^\ell q^m})$ is isomorphic to the following $H$-join:
	\[ H[K_{p^\ell q^m+p+q-1},E_{p^\ell-p},E_{q^m-q},E_{(p-1)(q-1)},E_{(q-1)(p^\ell-p)},E_{(p-1)(q^m-q)},E_{(p^\ell-p)(q^m-q)}],\] where $H$ is given in Figure~\ref{graph-H-of-cyclic-pmqn}. Consequently, $\Theta(D_{p^\ell q^m})$ is a $(6, 1)$-graph.
\end{proposition}

\begin{proof}[\textbf{Proof}]
Assume that $\ell, m\geq 2$. Set $n = p^\ell q^m$, and define		
	\begin{align*}
		S_1&= S(D_n) & S_5&=\{g\in D_n \colon |g|=p^jq,\ 2\leq j \leq \ell\}\\
		S_2& =\{g\in D_n \colon |g|=p^j,\ 2\leq j \leq  \ell\} & S_6&=\{g\in D_n \colon |g|=pq^j,\ 2\leq j \leq m\}\\
		S_3&=\{g\in D_n \colon |g|=q^j,\ 2\leq j \leq  m\} & S_7&=\{g\in D_n \colon |g|=p^jq^k,\ 2\leq j \leq \ell,\ 2\leq k \leq m\}.\\ 
		S_4&=\{g\in D_n \colon |g|=pq\}
	\end{align*}
By construction, the collection $\{S_1, S_2, \ldots, S_7\}$ forms a partition of $D_n$. By combining Theorems~\ref{th_numb-of-elem-cyclic-grp} and \ref{th_euler-phi-formula}, we obtain
$|S_1| = n+p+q-1$, $|S_2|=p^\ell-p$, $|S_3|=q^m-q$, $|S_4|=(p-1)(q-1)$, $|S_5|=(q-1)(p^\ell-p)$, $|S_6|=(p-1)(q^m-q)$, and $|S_7|=(p^\ell-p)(q^m-q)$. 
Now observe for each $i\in \{2,3,\ldots, 7\}$ that the gcd of the orders of any two distinct elements in $S_i$ is a composite number. Therefore $\Gamma_2 \colonequals \Theta(D_n)[S_2] \cong E_{p^\ell-p}$, $\Gamma_3 \colonequals \Theta(D_n)[S_3] \cong E_{q^m-q}$, $\Gamma_4 \colonequals\Theta(D_n)[S_4] \cong E_{(p-1)(q-1)}$, $\Gamma_5 \colonequals \Theta(D_n)[S_5] \cong E_{(q-1)(p^\ell-p)}$, $\Gamma_6 \colonequals \Theta(D_n)[S_6] \cong E_{(p-1)(q^m-q)}$, and $\Gamma_7 \colonequals \Theta(D_n)[S_7] \cong E_{(p^\ell-p)(q^m-q)}$. By Theorem~\ref{th_singletion-dominating-set}, we thus obtain $\Gamma_1 \colonequals \Theta(D_n)[S_1] \cong  K_{n+p+q-1}$ and $\Gamma_1 \sim \Gamma_i$
for each $i\in \{2,3,\ldots,7\}$. Furthermore, observe that $\Gamma_3 \sim \Gamma_4$, $\Gamma_3 \sim \Gamma_5$, and $\Gamma_2 \sim \Gamma_i$ for each $i\in \{3,4,6\}$. Finally, observe that $\Gamma_7$ is independent of $\Gamma_i$ for each $i \in \{2,3,\ldots,6\}$; $\Gamma_6$ is independent of $\Gamma_i$ for each $i \in \{3,4,5\}$; and $\Gamma_5$ is independent of both $\Gamma_2$ and $\Gamma_4$. Hence $\Theta(D_n) \cong H[\Gamma_1, \Gamma_2, \ldots, \Gamma_7]$, where $H$ is given in Figure~\ref{graph-H-of-cyclic-pmqn}.
\end{proof}

\begin{proposition}
We have \[\Theta(D_{pqr})\cong H[K_{pqr+p+q+r-2},E_{(p-1)(q-1)},E_{(q-1)(r-1)},E_{(p-1)(r-1)}, E_{(p-1)(q-1)(r-1)}],\] where $H$ is given in Figure~\ref{graph-H-of-cyclic-pqr}. Consequently, $\Theta(D_{pqr})$ is a $(4,1)$-graph.
\end{proposition}

\begin{proof}[\textbf{Proof}]
Set $n= pqr$, and define
	\begin{align*}
		S_1 & =S(D_n)& S_4&=\{g\in D_n \colon |g|=pr\}\\
		S_2&=\{g\in D_n \colon |g|=pq\} & S_5&=\{g\in D_n \colon |g|=pqr\}.\\
		S_3&=\{g\in D_n \colon |g|=qr\} 
	\end{align*}
By construction, the collection $\{S_1, S_2, \ldots, S_5\}$ forms a partition of $D_n$. By combining Theorems~\ref{th_numb-of-elem-cyclic-grp} and \ref{th_euler-phi-formula}, we obtain $|S_1| = n+p+q+r-2$, $|S_2|=(p-1)(q-1)$, $|S_3|=(q-1)(r-1)$, $|S_4|=(p-1)(r-1)$, and $|S_5|=(p-1)(q-1)(r-1)$. Now observe for each $i\in \{2,3,4,5\}$ that the gcd of the orders of any two distinct elements in $S_i$ is a composite number. Therefore $\Gamma_2 \colonequals \Theta(D_n)[S_2] \cong E_{(p-1)(q-1)}$, $\Gamma_3 \colonequals \Theta(D_n)[S_3] \cong E_{(q-1)(r-1)}$, $\Gamma_4 \colonequals \Theta(D_n)[S_4] \cong E_{(p-1)(r-1)}$, and $\Gamma_5 \colonequals \Theta(D_n)[S_5] \cong E_{(p-1)(q-1)(r-1)}$. By Theorem~\ref{th_singletion-dominating-set}, we thus obtain $\Gamma_1 \colonequals\Theta(D_n)[S_1] \cong K_{n+p+q+r-2}$ and $\Gamma_1 \sim \Gamma_i$
for each $i\in \{2,3,4,5\}$. Furthermore, observe that $\Gamma_2 \sim \Gamma_3$, $\Gamma_2 \sim \Gamma_4$, and $\Gamma_3 \sim \Gamma_4$. Finally, observe that $\Gamma_5$ is independent of each of the subgraphs $\Gamma_2$, $\Gamma_3$, and $\Gamma_4$. Hence 
$\Theta(D_n)=H[\Gamma_1, \Gamma_2, \ldots, \Gamma_5]$, where $H$ is given in Figure~\ref{graph-H-of-cyclic-pqr}.
\end{proof}

\subsection{Dicyclic Groups} For every integer $n\ge 2$, recall that the \emph{dicyclic group} $Q_{n}$ is a nonabelian group of order $4n$ having the presentation: \[Q_{n} = \langle a, b \colon a^{2n} = 1,\ a^n= b^2,\ b^{-1}ab = a^{-1}\rangle,\] where $1$ is the identity of $Q_{n}$ (cf.\cite[p~178]{james-b01}). If $n$ is a power of $2$, the $Q_n$ is called the \emph{generalized quaternion group} of order $4n$.

\vspace{0.5mm}
In this subsection, we prove that the graph $\Theta(Q_n)$ admits the structure of a $(k, 1)$-graph for each $n \in \{p, 2p, pq, 2^m, p^m\}$, where $p$ and $q$ are odd and $m \geq 2$. We assume throughout this subsection that $p$ and $q$ are distinct odd primes.


\begin{proposition}
We have $\Theta(Q_p)\cong  C_3[K_{p+1}, E_{p-1}, E_{2p}]$. Consequently, $\Theta(Q_p)$ is a $(2,1)$-graph.
\end{proposition}

\begin{proof}[\textbf{Proof}]
Define
	\begin{align*}
		S_1 &= S(Q_p) & S_2 &=\{g\in \langle a \rangle  \colon |g|= 2p\} & 
		S_3 &=  Q_p \setminus \langle a \rangle.
	\end{align*}
By construction, the collection $\{S_1, S_2, S_3\}$ forms a partition of $Q_p$.
Note that $|S_3| = 2p$. By combining Theorems~\ref{th_numb-of-elem-cyclic-grp} and \ref{th_euler-phi-formula}, we obtain $|S_1| = p+1$ and $|S_2| = p-1$. Now observe for each $i\in \{2, 3\}$ that the gcd of the orders of any two distinct elements in $S_i$ is a composite number. Therefore $\Gamma_2 \colonequals \Theta(Q_p)[S_2] \cong E_{p-1}$ and $\Gamma_3 \colonequals \Theta(Q_p)[S_3] \cong E_{2p}$. By Theorem~\ref{th_singletion-dominating-set}, we thus obtain $\Gamma_1 \colonequals \Theta(Q_p)[S_1] \cong  K_{p+1}$, $\Gamma_1 \sim \Gamma_2$, and $\Gamma_1 \sim \Gamma_3$. Furthermore,  observe that $\Gamma_2 \sim \Gamma_3$. Hence $\Theta(Q_p)\cong  C_3[\Gamma_1, \Gamma_2, \Gamma_3]$. 
\end{proof}

\begin{proposition}\label{pr_graph-H-of-Dic-2p}
We have $\Theta(Q_{2p})\cong H[K_{p+1}, E_{p-1},E_{2p-2},E_{4p+2}]$, where $H$ is given in Figure~\ref{graph-H-of-Dic-2p}. Consequently, $\Theta(Q_{2p})$ is a $(3,1)$-graph.
	\vspace{-2mm}
\begin{figure}[h]
	\centering
	\begin{tikzpicture}
		\draw[fill=black] (0.0,0.35) circle (1.2pt);
		\draw[fill=black] (0.0,-0.35) circle (1.2pt);
		\draw[fill=black] (0.7,0) circle (1.2pt);
		\draw[fill=black] (1.2,0) circle (1.2pt);
		
		\draw[line width=0.6 pt] (1.2,0) -- (0.7,0) -- (0.0,0.35)--(0.0, -0.35)--(0.7,0);
		
		\node at (0.7,0.2) {$1$};
		\node at (-0.15, 0.4) {$2$};
		\node at (-0.15, -0.4) {$4$};
		\node at (1.2, 0.2) {$3$};
	\end{tikzpicture}
	\vspace{-2mm}
	\caption{Graph $H$}	\label{graph-H-of-Dic-2p}
\end{figure}
\end{proposition}

\begin{proof}[\textbf{Proof}]
Set $n= 2p$, and define
	\begin{align*}
		S_1 &= S(Q_n) & S_3 &=\{g\in \langle a \rangle \colon |g|= 4p\}\\
		S_2& =\{g\in \langle a \rangle \colon |g|= 2p\} &S_4& = (Q_{n} \setminus \langle a \rangle)\cup \{g\in \langle a \rangle \colon |g|=4\}.
	\end{align*}
By construction, the collection $\{S_1, S_2, S_3, S_4\}$ forms a partition of $Q_n$. Note that $|Q_{n} \setminus \langle a \rangle| = 4p$. By combining Theorems~\ref{th_numb-of-elem-cyclic-grp} and \ref{th_euler-phi-formula}, we obtain $|S_1| = p+1$, $|S_2| = p-1$, $|S_3|=2p-2$, and $|S_4| = 4p+2$. Now observe for each $i\in \{2,3,4\}$ that the gcd of the orders of any two distinct elements in $S_i$ is a composite number. Therefore $\Gamma_2 \colonequals \Theta(Q_{n})[S_2]\cong E_{p-1}$, $\Gamma_3 \colonequals \Theta(Q_{n})[S_3]\cong E_{2p-2}$, and $\Gamma_4 \colonequals \Theta(Q_{n})[S_4]\cong E_{4p+2}$. By Theorem~\ref{th_singletion-dominating-set}, we thus obtain $\Gamma_1 \colonequals \Theta(Q_n)[S_1] \cong  K_{p+1}$ and $\Gamma_1 \sim \Gamma_i$ for each $i\in \{2,3,4\}$. Furthermore, observe that $\Gamma_2 \sim \Gamma_4$. Finally, observe that $\Gamma_3$ is independent of both $\Gamma_2$ and $\Gamma_4$. Hence $\Theta(Q_{n})\cong H[\Gamma_1, \Gamma_2, \Gamma_3, \Gamma_4]$, where $H$ is given in Figure~\ref{graph-H-of-Dic-2p}.
\end{proof}


\begin{proposition}
We have
\[\Theta(Q_{pq})\cong H[K_{p+q}, E_{p-1}, E_{q-1}, E_{(p-1)(q-1)}, E_{(p-1)(q-1)}, E_{2pq}],\] where $H$ is given in Figure~\ref{graph-H-of-Dic-pq}. Consequently, $\Theta(Q_{pq})$ is a $(5,1)$-graph.
	\vspace{-2mm}	
\begin{figure}[h]
		\centering
		\begin{tikzpicture}[very thick]
			\draw[fill=black] (-0.5,0) circle (1.2pt);
			\draw[fill=black] (0,0.9) circle (1.2pt);
			\draw[fill=black] (0,-1) circle (1.2pt);
			\draw[fill=black] (4,0.9) circle (1.2pt);
			\draw[fill=black] (4,-1) circle (1.2pt);
			\draw[fill=black] (2,-0.5) circle (1.2pt);
			
			\draw[line width=0.6 pt] (-0.5,0)--(0,0.9)--(4,0.9)--(4,-1)--(0,-1)--(-0.5,0);
			\draw[line width=0.6 pt] (0,-1)--(2,-0.5)--(2,-0.5)--(4,-1);
			\draw[line width=0.6 pt] (0,0.9)--(4,-1);
			\draw[line width=0.6 pt] (0,-1)--(4,0.9);					
			\draw[line width=0.6 pt] (0,0.9)--(0,-1);
			\draw[line width=0.6 pt] (0,0.9)--(2,-0.5)--(4,0.9);

			\node at (-0.65,0) {$5$};
			\node at (-0.2,0.9) {$1$};
			\node at (-0.2,-1) {$6$};
			\node at (4.2,0.9) {$2$};
			\node at (4.2,-1) {$3$};
			\node at (2,-0.75) {$4$};
		\end{tikzpicture}
	\vspace{-2mm}
		\caption{Graph $H$}	\label{graph-H-of-Dic-pq}
	\end{figure}
\end{proposition}

\begin{proof}[\textbf{Proof}]
Set $n = pq$, and define 
\begin{align*}
		S_1 &= S(Q_n)  & S_4 &=\{g\in \langle a \rangle \colon |g|= n\}\\
		S_2 &=\{g\in \langle a \rangle \colon |g|= 2p\} & S_5 &=\{g\in \langle a \rangle \colon |g|= 2n\}\\
		S_3 &=\{g\in \langle a \rangle \colon |g|= 2q\} & S_6 &= Q_{n} \setminus \langle a \rangle. 
	\end{align*}
By construction, the collection $\{S_1, S_2, \ldots, S_6\}$ forms a partition of $Q_{n}$. Note that $|S_6| = 2n$. By combining Theorems~\ref{th_numb-of-elem-cyclic-grp} and \ref{th_euler-phi-formula}, we obtain $|S_1|=p+q$, $|S_2|=p-1$, $|S_3|=q-1$, and $|S_4|= |S_5|= (p-1)(q-1)$. Now observe for each $i\in \{2,3,\ldots,6\}$ that the gcd of the orders of any two distinct elements in $S_i$ is a composite number. Therefore $\Gamma_2 \colonequals \Theta(Q_{n})[S_2]\cong E_{p-1}$, $\Gamma_3 \colonequals\Theta(Q_{n})[S_3]\cong E_{q-1}$, $\Gamma_4 \colonequals\Theta(Q_{n})[S_4]\cong E_{(p-1)(q-1)}$, $\Gamma_5 \colonequals\Theta(Q_{n})[S_5]\cong E_{(p-1)(q-1)}$, and $\Gamma_6 \colonequals \Theta(Q_{n})[S_6]\cong E_{2n}$. By Theorem~\ref{th_singletion-dominating-set}, we thus obtain $\Gamma_1 \colonequals \Theta(Q_n)[S_1] \cong  K_{p+q}$ and $\Gamma_1 \sim \Gamma_i$ for each $i\in \{2,3,\ldots,6\}$. Furthermore, observe that $\Gamma_2 \sim \Gamma_3$, $\Gamma_2 \sim \Gamma_4$, $\Gamma_3 \sim \Gamma_4$, and $\Gamma_6 \sim \Gamma_i$ for each $i \in \{2,3,4,5\}$. Finally, observe that $\Gamma_5$ is independent of each of the subgraphs  $\Gamma_2$, $\Gamma_3$, and $\Gamma_4$. Hence
$\Theta(Q_{n})\cong H[\Gamma_1, \Gamma_2, \ldots, \Gamma_6]$, where $H$ is given in Figure~\ref{graph-H-of-Dic-pq}.
\end{proof}

\begin{proposition}\label{pr_H-join-of-generalized-quaternion-grp}
We have $\Theta(Q_{2^m})\cong K_2\vee E_{2^{m+2}-2}$. Consequently, $\Theta(Q_{2^m})$ is a split graph.
\end{proposition}

\begin{proof}[\textbf{Proof}]
Set $n = 2^m$. Observe that $|S(Q_{n})| =2$ and $|T(Q_{n})| = 4n-2$.  Furthermore, the gcd of the orders of any two distinct elements in $T(Q_{n})$ is a composite number, and so $\Theta(Q_{n})[T(Q_{n})]\cong E_{4n-2}$. By Theorem~\ref{th_singletion-dominating-set}, we thus obtain $\Theta(Q_{n})[S(Q_{n})]\cong K_2$ and $\Theta(Q_{n})[S(Q_{n})] \sim \Theta(Q_{n})[T(Q_{n})]$. Hence 
$\Theta(Q_{n})\cong K_2 \vee E_{4n-2}$.
\end{proof}

\begin{proposition}\label{pr_prime-coprime-graph-of-Qpk}
For all $m\ge 2$,
$\Theta(Q_{p^m})\cong H[K_{p+1}, E_{p-1},E_{p^m-p},E_{p^m-p}, E_{2p^m}]$, where $H$ is given in Figure~\ref{graph-H-of-Dic-pk}. Consequently, $\Theta(Q_{p^m})$ is a $(4,1)$-graph.
	\vspace{-2mm}
	\begin{figure}[h]
		\centering
		\begin{tikzpicture}
			\draw[fill=black] (0,0) circle (1.2pt);
			\draw[fill=black] (1.5,0) circle (1.2pt);
			\draw[fill=black] (1.5,1) circle (1.2pt);
			\draw[fill=black] (0,1) circle (1.2pt);
			\draw[fill=black] (-0.5,0.5) circle (1.2pt);

			\draw[line width=0.6 pt] (0,0)--(1.5,0)--(1.5,1)--(0,1)--(0,0)--(-0.5,0.5)--(0,1)--(1.5,0);	
			\draw[line width=0.6 pt] (0,0)--(1.5,1);
			
			\node at (-0.15,-0.1) {$1$};
			\node at (1.65,0) {$2$};
			\node at (1.65,1) {$3$};
			\node at (-0.15,1.1) {$5$};
			\node at (-0.65,0.5) {$4$};
			\end{tikzpicture}
		\vspace{-2mm}
		\caption{Graph $H$}	\label{graph-H-of-Dic-pk}		
	\end{figure}
\end{proposition}

\begin{proof}[\textbf{Proof}]
Assume that $m\ge 2$. Set $n = p^m$, and define
	\begin{align*}
		S_1 &=  S(Q_n)& S_4 &=\{g\in \langle a \rangle \colon |g|=2p^j,\ 2\leq j \leq m\}\\
		S_2 &=\{g\in \langle a \rangle \colon |g|= 2p\} & S_5 &=  Q_n\setminus \langle a \rangle.\\ 
		S_3 &=\{g\in \langle a \rangle \colon |g|=p^j,\   2\leq j\leq m\} 		
	\end{align*}
By construction, the collection $\{S_1, S_2, \ldots, S_5\}$ forms a partition of $Q_{n}$. Note that $|S_5| = 2n$. By combining Theorems~\ref{th_numb-of-elem-cyclic-grp} and \ref{th_euler-phi-formula}, we obtain $|S_1|=p+1$, $|S_2|=p-1$, and $|S_3|= |S_4| =n-p$. Now observe for each $i\in \{2,3,4,5\}$ that the gcd of the orders of any two distinct elements in $S_i$ is a composite number. Therefore $\Gamma_2\colonequals \Theta(Q_n)[S_2]\cong E_{p-1}$, $\Gamma_3\colonequals\Theta(Q_n)[S_3]\cong E_{n-p}$, $\Gamma_4\colonequals\Theta(Q_n)[S_4]\cong E_{n-p}$, and $\Gamma_5\colonequals\Theta(Q_n)[S_5]\cong E_{2n}$. By Theorem~\ref{th_singletion-dominating-set}, we thus obtain $\Gamma_1\colonequals \Theta(Q_{n})[S_1] \cong K_{p+1}$ and $\Gamma_1 \sim \Gamma_i$ for each $i\in \{2,3,4,5\}$. Furthermore, observe that $\Gamma_2 \sim \Gamma_3$ and $\Gamma_5\sim \Gamma_i$ for each $i \in \{2,3,4\}$. Finally, observe that $\Gamma_4$ is independent of both $\Gamma_2$ and $\Gamma_3$. Hence $\Theta(Q_n)\cong H[\Gamma_1, \Gamma_2, \ldots, \Gamma_5]$, where $H$ is given in Figure~\ref{graph-H-of-Dic-pk}.
\end{proof}

\section{Hamiltonian Graphs}
For distinct primes $p, q$ with $p < q$, Adhikari and Banerjee \cite[Theorem~3.10]{adhikari-nsjom22} proved that $\Theta(\mathbb{Z}_{pq})$ is Hamiltonian if and only if $p = 2$. Motivated by this fact, we determine when the graph $\Theta(\mathbb{Z}_{n})$ is Hamiltonian. Furthermore, we determine when each of the graphs $\Theta(D_{n})$ and $\Theta(Q_{n})$ is Hamiltonian. We begin by recalling two fundamental properties of Hamiltonian graphs.

\begin{proposition}\label{pr_necessary-for-Hamiltonian}\cite[Proposition 7.2.3]{west-b01} 
	If $\Gamma$ is Hamiltonian graph, then $c(\Gamma-S) \leq |S|$ for all nonempty proper subsets $S$ of $V(\Gamma)$. 
\end{proposition}


\begin{theorem}\label{dirak_theorem}\cite[Theorem 7.2.8]{west-b01} 
	If $\Gamma$ is a graph of order $n\geq 3$ with minimum degree $\delta(\Gamma)\geq n/2$, then $\Gamma$ is Hamiltonian. 
\end{theorem}


The following theorem characterizes when the graph $\Theta(\mathbb{Z}_n)$ is Hamiltonian.

\begin{theorem}
	The graph $\Theta(\mathbb{Z}_{n})$ is Hamiltonian if and only if $n \in \{4, p, 2p\}$, where $p$ is an odd prime.
\end{theorem}

\begin{proof}[\textbf{Proof}]
Assume that $n \in \{4, p, 2p\}$, where $p$ is an odd prime. If $n=4$, then $\Theta(\mathbb{Z}_n)\cong K_2\vee E_2$. Therefore $\delta\big(\Theta(\mathbb{Z}_n)\big)=n/2$, and hence $\Theta(\mathbb{Z}_n)$ is Hamiltonian by Theorem~\ref{dirak_theorem}. If $n=p$, then $\Theta(\mathbb{Z}_n)\cong K_n$ by Theorem~\ref{th_group-EPO-iff-complete}. Since every complete graph with at least three vertices is Hamiltonian, it follows that $\Theta(\mathbb{Z}_n)$ is Hamiltonian. If $n=2p$, then $\Theta(\mathbb{Z}_{n})\cong K_{p+1}\vee E_{p-1}$ by Proposition~\ref{pr_group-of-2-distnct-prime}. Therefore $\delta\big(\Theta(\mathbb{Z}_{n})\big)=p+1> n/2$, and hence $\Theta(\mathbb{Z}_{n})$ is Hamiltonian by Theorem~\ref{dirak_theorem}.
	
	\vspace{0.4mm}
	To prove the converse, we will show that its contrapositive holds true. Assume that $n\notin \{4, p, 2p\}$, where $p$ is an odd prime. The graph $\Theta(\mathbb{Z}_n)$ is non-Hamiltonian for $n = 1$ and $n=2$. If $n = 12$, then $\Theta(\mathbb{Z}_n) \cong  H[K_4, E_2,E_2,E_4]$ by Proposition~\ref{pr_H-join-cyclic-pqm}, where $H$ is given in Figure~\ref{graph-H-of-cyclic-pqm}. Since $c\big(H[K_4, E_2,E_2,E_4] - V(K_4)\big) > |V(K_4)|$, it follows from Theorem~\ref{dirak_theorem} that $\Theta(\mathbb{Z}_n)$ is non-Hamiltonian. Next if 
	$n = 30$, then $\Theta(\mathbb{Z}_n) \cong  H[K_8, E_2, E_8, E_4, E_8]$ by Proposition~\ref{pr_H-join-cyclic-pqr}, where $H$ is given in Figure \ref{graph-H-of-cyclic-pqr}. Since $c\big(H[K_8, E_2, E_8, E_4, E_8] - V(K_8)\big) > |V(K_8)|$, it follows from Theorem~\ref{dirak_theorem} that 
	 $\Theta(\mathbb{Z}_n)$ is non-Hamiltonian.

	\vspace{0.4mm}	
	So, assume further that $n$ is distinct from $12$ and $30$. Let $p_1, p_2, \ldots, p_k$ be all distinct prime factors of $n$, and let $S_n\colonequals\{g\in \mathbb{Z}_n\colon |g|=n\}$. By combining Theorems~\ref{th_numb-of-elem-cyclic-grp} and \ref{th_euler-phi-formula}, we obtain $|S(\Z_n)|=1+\sum_{i=1}^{k}\varphi(p_i)$ and $|S_n|=\varphi(n)$. Moreover, observe that $\Theta(\mathbb{Z}_n)[S_n]\cong E_{\varphi(n)}$ and the subgraph $\Theta(\Z_n)-S(\Z_n)$ is independent of the subgraph $\Theta(\Z_n)[S_n]$. This implies that $c\big(\Theta(\Z_n)-S(\Z_n)\big)\geq \varphi(n)$. By combining Propositions~5 and 6 from \cite{shubh_rg25}, we have  $\varphi(n)>1+\sum_{i=1}^{k}\varphi(p_i)$. Therefore, since $|S(\Z_n)| = 1+\sum_{i=1}^{k}\varphi(p_i)$, we obtain $c\big(\Theta(\Z_n)-S(\Z_n)\big) > |S(\Z_n)|$. Hence $\Theta(\Z_n)$ is non-Hamiltonian by Proposition~\ref{pr_necessary-for-Hamiltonian}. 	
\end{proof}

The theorem below asserts that the graph $\Theta(D_n)$ is Hamiltonian for all $n$.

\begin{theorem}
	The graph $\Theta(D_n)$ is Hamiltonian.
\end{theorem}

\begin{proof}[\textbf{Proof}]
	Note that $D_n$ contains at least $n$ involutions, and so $|S(D_n)|\geq n+1$. Therefore, by Theorem~\ref{th_singletion-dominating-set} we  obtain $\delta\big(\Theta(D_n)\big)\geq n+1$. Hence, since $|D_n| = 2n$, the graph $\Theta(D_n)$ is Hamiltonian by Theorem~\ref{dirak_theorem}.
\end{proof}


The theorem below characterizes when the graph $\Theta(Q_n)$ is Hamiltonian.

\begin{theorem}
The graph $\Theta(Q_n)$ is Hamiltonian if and only if $n$ is odd.
\end{theorem}

\begin{proof}[\textbf{Proof}]
Assume that $n$ is odd. Note that $\langle a \rangle$ is a cyclic subgroup of order $2n$ in $Q_n$. Therefore $\Theta(Q_n)[\langle a \rangle]\cong \Theta(\Z_{2n})$.
Furthermore, since $|Q_n\setminus\langle a \rangle| = 2n$ and each element in $Q_n\setminus\langle a \rangle$ has order $4$, we obtain $\Theta(Q_n)[Q_n\setminus\langle a \rangle]\cong E_{2n}$. Since $n$ is odd, we obtain $\gcd(4, d) \in \{1, 2\}$ for each positive divisor $d$ of $2n$. Therefore $\Theta(Q_n)[Q_n\setminus \langle a \rangle] \sim \Theta(Q_n)[\langle a \rangle]$, and so $\Theta(Q_n)\cong \Theta(\Z_{2n}) \vee E_{2n}$. This implies that $\delta\big(\Theta(Q_n)\big)= 2n$. Hence, since $|Q_n|= 4n$, the graph $\Theta(Q_n)$ is Hamiltonian by Theorem~\ref{dirak_theorem}.  

\vspace{0.4mm}
	
To prove the converse part, we will show that its contrapositive holds true. Assume that $n$ is even. If $n= 2$, then $\Theta(Q_2)\cong K_2\vee E_6$ by Proposition~\ref{pr_H-join-of-generalized-quaternion-grp}. Since
	$c\big((K_2\vee E_6)-V(K_2)\big) = 6 > |V(K_2)|$, it follows from Proposition~\ref{pr_necessary-for-Hamiltonian} that $\Theta(Q_n)$ is non-Hamiltonian. Next if $n= 6$, then $\Theta(Q_6)\cong H[K_4,E_2,E_4,E_{14}]$ by Proposition~\ref{pr_graph-H-of-Dic-2p}, where $H$ is given in Figure~\ref{graph-H-of-Dic-2p}. Since 
	$c\big(H[K_{4},E_{2},E_{4},E_{14}]-V(K_4)\big) = 5 > |V(K_4)|$, it follows from Proposition~\ref{pr_necessary-for-Hamiltonian} that $\Theta(Q_n)$ is non-Hamiltonian.

	\vspace{0.4mm}
So, assume further that $n$ is distinct from $2$ and $6$. Let $p_1, p_2, \ldots, p_k$ be all distinct prime factors of $n$, and let $S_{2n}\colonequals\{g\in Q_n\colon |g|=2n\}$. By combining Theorems~\ref{th_numb-of-elem-cyclic-grp} and \ref{th_euler-phi-formula}, we obtain $|S(Q_n)|=1+\sum_{i=1}^{k}\varphi(p_i)$ and $|S_{2n}|=\varphi(2n)$. Furthermore, observe that $\Theta(Q_n)[S_{2n}] \cong E_{\varphi(2n)}$ and the subgraph $\Theta(Q_n)-S(Q_n)$ is independent of the subgraph $\Theta(Q_n)[S_{2n}]$. Therefore $c\big(\Theta(Q_n)-S(Q_n)\big)\geq \varphi(2n)$. By \cite[Proposition~6]{shubh_rg25}, we have $\varphi(2n)>1+\sum_{i=1}^{k}\varphi(p_i)$. Thus, since $|S(Q_n)| = 1+\sum_{i=1}^{k}\varphi(p_i)$, we obtain $c\big(\Theta(Q_n)-S(Q_n)\big) > |S(Q_n)|$. Hence $\Theta(Q_n)$ is non-Hamiltonian by Proposition~\ref{pr_necessary-for-Hamiltonian}. 
\end{proof}


\section{Clique Number}
In this section, we determine the clique number of each of the graphs $\Theta(\Z_n)$, $\Theta(D_n)$, and $\Theta(Q_n)$. Recall that the \emph{clique number} $\omega(\Gamma)$ of a graph $\Gamma$ is the cardinality of its maximum clique.


\vspace{0.4mm}

In the following theorem, we determines the clique number of the graph $\Theta(\Z_n)$ for all $n>1$.

\begin{theorem}\label{clique_Theta_mZn}
Let $n>1$ be an integer with canonical factorization  $n=p_1^{r_1}p_2^{r_2}\cdots p_k^{r_k}$. Then
	\vspace{-3mm}
	\begin{equation*}
		\omega\big(\Theta(\Z_n)\big)=1+\sum_{i=1}^{k}(p_i-1)+\displaystyle\sum_{\substack{i=1 \\p_i^2|n}}^k 1+ \frac{k(k-1)}{2}.
	\end{equation*}
\end{theorem}

\begin{proof}[ \textbf{Proof}]
The result follows directly from Theorem~\ref{th_group-EPO-iff-complete} when $n$ is prime. So, assume further that $n$ is composite. Let $W$ be a maximum clique of $\Theta(\Z_n)$. By Theorem~\ref{th_singletion-dominating-set}, it follows that $S(\Z_n)\subseteq W$. For each composite divisor $d$ of $n$, let $S_d$ denote the set of all elements of order $d$ in $\Z_n$. By Theorem~\ref{th_numb-of-elem-cyclic-grp}, we have $|S_d| = \varphi(d)$, and consequently, $\Theta(\Z_n)[S_d]\cong E_{\varphi(d)}$. Therefore $|W\cap S_d|\leq 1$. Furthermore, observe that the gcd of any two distinct products, each of them formed by multiplying two primes, is either one or a prime. Therefore, for all distinct indices $i, j$ with $1 \leq i < j \leq k$, it follows that $|W\cap S_{p_ip_j}|=1$, and $|W\cap S_{p_i^2}|=1$ whenever $p_i^2\mid n$. 
	
	\vspace{0.4mm}	
We now proceed to show that $W\cap S_d=\varnothing$ whenever $d$ is neither the square of a prime nor the product of two distinct primes. We consider the following three distinct cases:
	
	\vspace{0.4mm}	
	\noindent\textbf{Case 1:} $p_i^3\mid d$ for some prime $p_i$. Then $\gcd(d,p_i^2)=p_i^2$, and consequently no vertex of $\Theta(\Z_n)[S_d]$ is adjacent to the vertex in $W$ corresponding to an element of order $p_i^2$. Therefore $W\cap S_d =\varnothing$.

	\vspace{0.4mm}	
	\noindent\textbf{Case 2:} $p_i^ap_j^b \mid d$ for some distinct primes $p_i, p_j$, where $a,b\in \N$ with $a+b=3$. Then $\gcd(d,p_ip_j)=p_ip_j$, and consequently no vertex of $\Theta(\Z_n)[S_d]$ is adjacent to the vertex in $W$ corresponding to an element of order $p_ip_j$. Therefore $W\cap S_d =\varnothing$.

	\vspace{0.4mm}	
	\noindent\textbf{Case 3:} $p_ip_jp_\ell \mid d$ for some distinct primes $p_i, p_j, p_\ell$. Then $\gcd(d,p_ip_j)=p_ip_j$, and consequently no vertex of $\Theta(\Z_n)[S_d]$ is adjacent to the vertex in $W$ corresponding to an element of order $p_ip_j$. Therefore $W\cap S_d =\varnothing$.
	
	\vspace{0.4mm}	
	Thus
	\[W=S(\Z_n)\cup\bigg( \bigcup_{\substack{i=1 \\p_i^2|n}}^k\big(W\cap S_{p_i^2}\big)\bigg)\cup\bigg( \bigcup_{\begin{subarray}{l}
			i,j=1\\
			i<j
	\end{subarray}}^k \big(W\cap S_{p_ip_j}\big)\bigg).\]
Moreover, observe that   
$\big|\bigcup_{\begin{subarray}{l}
		i,j=1\\
		i<j
\end{subarray}}^k\big(W\cap S_{p_ip_j}\big)\big|=\frac{k(k-1)}{2}$. Also, we obtain $|S(\Z_n)|=1+\sum_{i=1}^{k}(p_i-1)$ by combining Theorems~\ref{th_numb-of-elem-cyclic-grp} and \ref{th_euler-phi-formula}. Hence \[ \omega\big(\Theta(\Z_n)\big)=|W|=1+\sum_{i=1}^{k}(p_i-1)+\displaystyle\sum_{\substack{i=1 \\p_i^2|n}}^k 1+ \frac{k(k-1)}{2}. \]
\end{proof}

In order to determine the clique number of the graph $\Theta(D_n)$, we first prove the following lemma.

\begin{lemma}\label{le_join-for-Dn}
We have $\Theta(D_n)\cong  \Theta(\Z_n) \vee K_n$.
\end{lemma}

\begin{proof}[\textbf{Proof}]
	Note that $\langle a \rangle$ is a cyclic subgroup of order $n$ in $D_n$. Therefore $\Theta(D_n)[\langle a \rangle]\cong \Theta(\Z_n)$. Furthermore, since $|D_n\setminus\langle a \rangle| = n$ and each element in $D_n\setminus\langle a \rangle$ has order $2$, we obtain $\Theta(D_n)[D_n\setminus\langle a \rangle]\cong K_n$ and $\Theta(D_n)[D_n\setminus\langle a \rangle]\sim \Theta(D_n)[\langle a \rangle]$. Hence $\Theta(D_n)\cong  \Theta(\Z_n) \vee K_n$.
\end{proof}


\begin{corollary}
The clique number of the graph $\Theta(D_n)$ is given by
\[\omega\big(\Theta(D_n)\big)=n+\omega\big(\Theta(\Z_n)\big).\]
\end{corollary}
\begin{proof}[\textbf{Proof}]
By Lemma~\ref{le_join-for-Dn}, the graph $\Theta(D_n)$ is isomorphic to $\Theta(\Z_n) \vee K_n$. Therefore 
\[\omega\big(\Theta(D_n)\big)=\omega\big(\Theta(\Z_n) \vee K_n\big) = \omega\big(\Theta(\Z_n) \big) +  \omega(K_n)= \omega\big(\Theta(\Z_n)\big)+n.\]
\end{proof}


The following proposition expresses the clique number of the graph $\Theta(Q_n)$ in terms of the clique number of $\Theta(\mathbb{Z}_{2n})$.

\begin{proposition}
The clique number of the graph $\Theta(Q_n)$ is given by
	\[\omega\bigl(\Theta(Q_n)\bigr)=\begin{cases}
		\omega\bigl(\Theta(\mathbb{Z}_{2n})\bigr) & \text{if } n \text{ is even},\\
		1+\omega\bigl(\Theta(\mathbb{Z}_{2n})\bigr) & \text{if } n \text{ is odd}.
	\end{cases}\]
\end{proposition}
\begin{proof}[\textbf{Proof}]
Note that $\langle a \rangle$ is a cyclic subgroup of order $2n$ in $Q_n$. Therefore $\Theta(Q_n)[\langle a \rangle]\cong \Theta(\mathbb{Z}_{2n})$, and so $\omega\big(\Theta(\mathbb{Z}_{2n})\big)\leq \omega\big(\Theta(Q_n)\big)$. Furthermore, since $|Q_n\setminus\langle a \rangle|=2n$ and each element in $Q_n\setminus\langle a \rangle$ has order $4$, we obtain $\Theta(Q_n)[Q_n\setminus\langle a \rangle]\cong E_{2n}$. We now proceed by considering two distinct cases:

	\vspace{0.4mm}
	\noindent\textbf{Case 1:}  $n$ is odd. Then for each positive divisor $d$ of $2n$, we have $\gcd(4, d)\in \{1,2\}$. Since each element in $Q_n\setminus\langle a \rangle$ has order $4$, it follows that $\Theta(Q_n)[\langle a \rangle] \sim  \Theta(Q_n)[Q_n\setminus\langle a \rangle]$. Therefore $\Theta(Q_n)\cong  \Theta(\mathbb{Z}_{2n}) \vee E_{2n}$. Hence
	$\omega\big(\Theta(Q_n)\big)= \omega\big(\Theta(\mathbb{Z}_{2n})\big) + \omega(E_{2n})=
	\omega\big(\Theta(\mathbb{Z}_{2n})\big)+1$.

	\vspace{0.4mm}
	\noindent\textbf{Case 2:} $n$ is even. Then $4 \mid 2n$, and so $\langle a \rangle$ contains $\varphi(4) = 2$ elements of order $4$ by Theorem~\ref{th_numb-of-elem-cyclic-grp}. Let $W$ be a maximum clique of the subgraph $\Theta(Q_n)[\langle a \rangle]$. By Theorem~\ref{clique_Theta_mZn}, the clique $W$ contains exactly one vertex $v$ corresponding to an element of order $4$ in $\langle a \rangle$. Furthermore, since each element in $Q_n\setminus\langle a \rangle$ has order $4$, the vertex $v$ is non-adjacent to every vertex of the subgraph $\Theta(Q_n)[Q_n \setminus \langle a \rangle]$. Therefore $W$ is a maximum clique of $\Theta(Q_n)$. Hence $\omega\bigl(\Theta(Q_n)\bigr)=|W|=\omega\bigl(\Theta(Q_n)[\langle a \rangle]\bigr)= \omega\bigl(\Theta(\mathbb{Z}_{2n})\bigr)$.
\end{proof}


\section{Vertex Degree}

By Theorem~\ref{th_singletion-dominating-set}, every element $g$ of the subset $S(G)$ of a group $G$ has degree $\deg_{\Theta(G)}(x)= |G|-1$ in the graph $\Theta(G)$.

\vspace{0.4mm}
In this section, we determine the degrees of \emph{composite-order vertices} of each of the graphs $\Theta(\Z_n)$, $\Theta(D_n)$, and $\Theta(Q_n)$, where a \emph{composite-order vertex} of the graph $\Theta(G)$ is one that represents
an element of composite order in the group $G$.

\vspace{0.4mm}
In order to determine the degree of any vertex of the graph $\Theta(\Z_n)$, we require the following lemma.

\begin{lemma}\label{Le_sum_phi(d)}
Let $n > 1$ be an integer with canonical factorization $n=\prod_{i=1}^kp_i^{\alpha_i}$. Then \[\displaystyle \sum_{d\mid n}\varphi(d)=\displaystyle\sum_{\substack{a_i\in \{0,1\}\\1 \leq i \leq k}}\bigg(\prod_{i=1}^{k}(p_i^{\alpha_i}-1)^{a_i}\bigg),\]
where $d$ runs through all positive divisors of $n$. 
\end{lemma}

\begin{proof}[\textbf{Proof}]

Note that each positive divisor $d$ of $n$ admits a representation of the form 
$d =\prod_{i=1}^{k} p_i^{a_i\beta_i}$, where $a_i \in \{0,1\}$ and $1\leq \beta_i\leq \alpha_i$ for each $i \in\{ 1,2,\ldots, k\}$ (cf. \cite[Theorem~6.1]{burton-b07}). By Theorem~\ref{th_euler-phi-formula}, we therefore obtain
	\[\varphi(d) =\varphi\bigg(\prod_{i=1}^{k} p_i^{a_i\beta_i}\bigg)= \prod_{i=1}^{k} \varphi\Big( p_i^{a_i\beta_i}\Big)= \prod_{i=1}^{k} \Big(\varphi\big( p_i^{\beta_i}\big)\Big)^{a_i}=
	\prod_{i=1}^{k}\Big(p_i^{\beta_i}-p_i^{\beta_i-1}\Big)^{a_i}.
	\] 
Hence
\begin{align*}
	\sum_{d\mid n} \varphi(d)&= \sum_{\substack{a_i\in \{0,1\} \\ 1\leq i\leq k}}\Bigg(\sum_{\substack{1\leq \beta_i\leq \alpha_i\\ 1\leq i\leq k}} \bigg(\prod_{i = 1}^k \big(p_i^{\beta_i}-p_i^{\beta_i-1}\big)^{a_i}\bigg)\Bigg) \\
	&= \sum_{\substack{a_i\in \{0,1\} \\ 1\leq i\leq k}}\Bigg(\prod_{i = 1}^k \bigg(\sum_{1\leq \beta_i\leq \alpha_i} \big(p_i^{\beta_i}-p_i^{\beta_i-1}\big)^{a_i}\bigg)\Bigg) \\
	&= \sum_{\substack{a_i\in \{0,1\} \\ 1\leq i\leq k}} \bigg(\prod_{i = 1}^k\Big( p_i^{\alpha_i}-1\Big)^{a_i} \bigg).
\end{align*}	

\end{proof}


\begin{theorem}\label{th_vert_deg_Zn}
Let $n$ be a composite number with canonical factorization $n=\prod_{i=1}^kp_i^{\alpha_i}$, and let $x\in \mathbb{Z}_n$ 
with composite order $\prod_{i=1}^{k}{p_i}^{\beta_i}$. Then 
	\vspace{-1mm}
	\begin{equation*}
		\deg_{\Theta(\Z_n)}(x)=
		\sum_{\Delta}\prod_{i=1}^k(p_i^{\gamma_i}-1)^{a_i},
		\vspace{-4mm}
	\end{equation*}
	where $\gamma_i=\begin{cases}
		1& \text{if } \beta_i \geq 2,\\
		\alpha_i &\text{if } \beta_i\leq 1,
	\end{cases}$
	and	$\Delta$ ranges over all choices of $a_i \in \{0,1\}$ satisfying
	 $\displaystyle\sum_{\substack{1\leq j\leq k\\ \beta_j\geq 1}}a_j\leq 1$.
\end{theorem} 
\begin{proof}[\textbf{Proof}]
Clearly $0\leq\beta_i\leq\alpha_i$ for each $i\in\{1,2,\ldots,k\}$. Observe for each vertex $y$ of $\Theta(\Z_n)$ that $xy\in E\big(\Theta(\Z_n)\big)$ if and only if the orders $|x|$ and $|y|$ share at most one prime factor with exponent 1. Therefore, the vertex $x$ is adjacent to each vertex of $\Theta(\Z_n)$ whose order divides $d\colonequals\prod_{i=1}^{k}p_i^{a_i\gamma_i}$, where 
$\gamma_i=\begin{cases}
	1& \text{if } \beta_i \geq 2,\\
	\alpha_i &\text{if } \beta_i\leq 1,
\end{cases}$ and $a_i\in\{0,1\}$ such that
$\displaystyle\sum_{\substack{1\leq j\leq k\\\beta_j\geq 1}}a_j\leq 1$. By Theorem~\ref{th_numb-of-elem-cyclic-grp}, the group $\mathbb{Z}_n$ contains $\varphi(d')$ elements of order $d'$, for each positive divisor $d'$ of $n$. Therefore $\deg_{\Theta(\Z_n)}(x)=\sum_{d'\mid d}\varphi(d')$. Hence, by Lemma \ref{Le_sum_phi(d)} we obtain
	\begin{equation*}
		\deg_{\Theta(\Z_n)}(x)=
		\sum_{\Delta}\prod_{i=1}^k(p_i^{\gamma_i}-1)^{a_i},
		\vspace{-4mm}
	\end{equation*}

	where $\gamma_i=\begin{cases}
	1& \text{if } \beta_i \geq 2,\\
	\alpha_i &\text{if } \beta_i\leq 1,
\end{cases}$
and	$\Delta$ ranges over all choices of $a_i \in \{0,1\}$ satisfying
$\displaystyle\sum_{\substack{1\leq j\leq k\\ \beta_j\geq 1}}a_j\leq 1$.
\end{proof}

\begin{proposition}
If $x$ is an element of composite order in the group $D_n$, then
	$\deg_{\Theta(D_n)}(x)=
	n+ \deg_{\Theta(\Z_n)}(x)$.
\end{proposition}

\begin{proof}[\textbf{Proof}]
It follows directly from Lemma \ref{le_join-for-Dn}.
\end{proof}


\begin{proposition}
Let $x$ be an element of composite order $d$ in the group $Q_n$. Then
	\[\deg_{\Theta(Q_n)}(x)=\begin{cases}
		2n+\deg_{\Theta(\mathbb{Z}_{2n})}(x) & \text{if } 4\nmid d ,\\
		\deg_{\Theta(\mathbb{Z}_{2n})}(x)& \text{if } 4\mid d.
	\end{cases}\]
\end{proposition} 

\begin{proof}[\textbf{Proof}]
Note that $\langle a \rangle$ is a cyclic subgroup of order $2n$ in $Q_n$. Therefore $\Theta(Q_n)[\langle a \rangle]\cong \Theta(\Z_{2n})$.
Furthermore, since $|Q_n\setminus\langle a \rangle| = 2n$ and each element of $Q_n\setminus\langle a \rangle$ has order $4$, we obtain $\Theta(Q_n)[Q_n\setminus\langle a \rangle]\cong E_{2n}$. If $4\mid d$, then $\gcd(4, d) = 4$. Therefore, the vertex $x$ is non-adjacent to all vertices of $\Theta(Q_n)$ corresponding to elements of order $4$ in $Q_n$. Hence $\deg_{\Theta(Q_n)}(x)=\deg_{\Theta(\mathbb{Z}_{2n})}(x)$. So, assume that $4\nmid d$. Then $\gcd(4,d)\in\{1,2\}$. Therefore, the vertex $x$ is adjacent to all  vertices of $\Theta(Q_n)$ corresponding to elements of order $4$ in $Q_n$. Hence $\deg_{\Theta(Q_n)}(x)=2n+\deg_{\Theta(\mathbb{Z}_{2n})}(x)$.	
\end{proof}



\end{document}